\documentclass[reqno]{amsart}
\usepackage{amssymb}
\usepackage{tikz-cd}
\usepackage{array}
\usepackage{graphicx}
\usepackage{caption, booktabs}
\usepackage[all]{xy}
\usepackage{comment}
\usepackage{cite}
\usepackage[left=2.5cm,right=2.5cm,top=3cm,bottom=3cm]{geometry}
\usepackage{color}
\usepackage{comment}

\newcommand{\C}{\mathbb{C}}
\newcommand{\tr}{\mathrm{tr}}
\newcommand{\g}{\mathfrak{g}}

\newcommand{\lf}{\mathfrak{l}}
\newcommand{\uu}{\mathfrak{u}}
\newcommand{\h}{\mathfrak{h}}

\newcommand{\greg}{\mathfrak{g}_{\mathrm{reg}}}

\newcommand{\sln}{\mathfrak{sl}}

\newcommand{\gsing}{\mathfrak{g}_{\mathrm{sing}}}

\makeatletter
\@namedef{subjclassname@1991}{2020 Mathematics Subject Classification}
\makeatother

\numberwithin{equation}{section}

\theoremstyle{definition}
\newtheorem{thm}{Theorem}[section]

\newtheorem{lem}[thm]{Lemma}
\newtheorem{cor}[thm]{Corollary}
\newtheorem{prop}[thm]{Proposition}
\theoremstyle{definition}

\newtheorem*{thm*}{Theorem}

\begin{document}

\title[On the singularities of Mishchenko--Fomenko systems]{On the singularities of Mishchenko--Fomenko systems}

\author[Peter Crooks]{Peter Crooks}
\author[Markus R\"oser]{Markus R\"oser}
\address[Peter Crooks]{Department of Mathematics, Northeastern University, 360 Huntington Avenue, Boston, MA 02115, USA}
\email{p.crooks@northeastern.edu}
\address[Markus R\"oser]{Fachbereich Mathematik, Universit\"at Hamburg, 20146 Hamburg, Germany}
\email{markus.roeser@uni-hamburg.de}


\begin{abstract}
To each complex semisimple Lie algebra $\g$ and regular element $a\in\greg$, one associates a Mishchenko--Fomenko subalgebra $\mathcal{F}_a\subseteq\mathbb{C}[\g]$. This subalgebra amounts to a completely integrable system on the Poisson variety $\g$, and as such has a bifurcation diagram $\Sigma_a\subseteq\mathrm{Spec}(\mathcal{F}_a)$. We prove that $\Sigma_a$ has codimension one in $\mathrm{Spec}(\mathcal{F}_a)$ if $a\in\greg$ is not nilpotent, and that it has codimension one or two if $a\in\greg$ is nilpotent. In the nilpotent case, we show each of the possible codimensions to be achievable. Our results significantly sharpen existing estimates of the codimension of $\Sigma_a$.   
\end{abstract}

\maketitle

\vspace{-10pt}

{\small\tableofcontents} 

\section{Introduction}\label{Section: Introduction}
\subsection{Context and main results}
Mishchenko--Fomenko systems \cite{Mishchenko} represent a paradigm of Lie-theoretic approaches to complete integrability, and they have received considerable attention in both classical \cite{Mishchenko,Vinberg,Manakov,Bolsinov,Bolsinov91} and modern \cite{BolsinovRemarks,BolsinovOshemkov,AbeCrooks,Moreau,Molev,Panyushev,Panyushev2,CrooksRayan,CrooksRoeserDocumenta,IzosimovPhys} contexts. While such systems have many incarnations, they may be defined in the presence of a finite-dimensional, rank-$\ell$ complex semisimple Lie algebra $\mathfrak{g}$ with adjoint group $G$. The Killing form then induces a $G$-module isomorphism between $\mathfrak{g}$ and $\mathfrak{g}^*$, so that the canonical Lie--Poisson structure on $\g^*$ amounts to $\g$ being a Poisson variety. One next considers the locus of regular elements $\greg\subseteq\g$, i.e. the $G$-invariant, open, dense subvariety of all $x\in\g$ with $\ell$-dimensional centralizers $\g_x\subseteq\g$. Fixing such an element $a\in\greg$, each $f\in\mathbb{C}[\g]:=\mathrm{Sym}(\g^*)$ and $\lambda\in\mathbb{C}$ determine a polynomial $f_{\lambda,a}\in\mathbb{C}[\g]$ by the formula $$f_{\lambda,a}(x)=f(x+\lambda a),\quad x\in\mathfrak{g}.$$ The \textit{Mishchenko--Fomenko subalgebra} is then defined to be the subalgebra $\mathcal{F}_{a}\subseteq\mathbb{C}[\g]$ generated by all $f_{\lambda,a}$ with $f\in\mathbb{C}[\g]^G$ and $\lambda\in\mathbb{C}$. This is a maximal Poisson-commutative subalgebra of $\mathbb{C}[\g]$, for which one can choose $b:=\frac{1}{2}(\dim(\g)+\ell)$ homogeneous, algebraically independent generators. A more geometric interpretation is that the canonical map $$F_a:\g\longrightarrow\mathrm{Spec}(\mathcal{F}_a)\cong\mathbb{C}^b$$ forms a completely integrable system on the Poisson variety $\g$. The term \textit{Mishchenko--Fomenko system} is sometimes used in reference to $F_a$.     

At the same time, many interesting results in integrable systems theory are expressed in terms of critical points and critical values. One is thereby motivated to study the critical points and critical values of $F_a$, to which $\g_{\text{sing}}:=\mathfrak{g}\setminus\mathfrak{g}_{\text{reg}}$ turns out to be relevant. The critical points of $F_a$ are given by
$$\gsing^a:=\g_{\text{sing}}+\mathbb{C}a\subseteq\g$$ a consequence of Bolsinov's work \cite[Proposition 3.1]{Bolsinov} in the bi-Hamiltonian context. It follows that 
$$\Sigma_a:=F_a(\gsing^a)\subseteq\mathrm{Spec}(\mathcal{F}_a)$$ is the set of critical values, sometimes also known as the \textit{bifurcation diagram} of $F_a$. This object features prominently in several papers, some of which relate $\Sigma_a$ to spectral curves \cite{Brailov,Konyaev,Konyaev2,BrailovFomenko}. 

One recognizes that $\Sigma_a$ is a constructible subset of $\mathrm{Spec}(\mathcal{F}_a)$, and as such must contain an open dense subset of its closure $\overline{\Sigma_a}\subseteq\mathrm{Spec}(\mathcal{F}_a)$. It is therefore reasonable to define the dimension of $\Sigma_a$ to be that of the affine variety $\overline{\Sigma_a}$. In some sense, $\dim(\Sigma_a)$ measures the prevalence of the critical values of $F_a$. This gives context for \cite[Proposition 5.2]{CrooksRoeserDocumenta} and \cite[Remark 6.1]{CrooksRoeserDocumenta}, which we state together as follows.

\begin{thm*}\label{Theorem: Prelim} If $a\in\greg$, then $\Sigma_a$ has codimension one or two in $\mathrm{Spec}(\mathcal{F}_a)$. Each of these codimensions is achievable in examples.
\end{thm*}

This result has an obvious shortcoming; it does not address the dependence of $\dim(\Sigma_a)$ on $a\in\greg$. With this issue in mind, our main result is the following refinement of the previous theorem. 

\begin{thm}\label{Thm: Main Theorem}
Suppose that $a\in\greg$.
\begin{itemize}
\item[(i)] If $a$ is not nilpotent, then $\Sigma_a$ has codimension one in $\mathrm{Spec}(\mathcal{F}_a)$.
\item[(ii)] If $a$ is nilpotent, then $\Sigma_a$ has codimension one or two in $\mathrm{Spec}(\mathcal{F}_a)$. Each of these codimensions is achievable in examples. 
\end{itemize}
\end{thm}

Our manuscript is organized as follows. Section \ref{Section: Some ingredients} establishes the notation, conventions, and first results underlying the proof of Theorem \ref{Thm: Main Theorem}. Section \ref{Section: The proof} subsequently provides the proof itself, devoting \ref{Subsection: Proof i} and \ref{Subsection: Proof ii} to the proofs of (i) and (ii), respectively.  

\section{Some ingredients}\label{Section: Some ingredients}

\subsection{Lie-theoretic preliminaries}
Retain the notation and objects introduced in Section \ref{Section: Introduction}.  We will have occasion to consider the three integers $$\ell:=\mathrm{rank}(\g),\quad b:=\frac{1}{2}(\dim(\g)+\ell),\quad\text{and}\quad u:=b-\ell,$$ noting that $b$ (resp. $u$) is the dimension of any Borel subalgebra (resp. maximal nilpotent subalgebra) of $\g$.

Let $\langle\cdot,\cdot\rangle:\g\otimes_{\mathbb{C}}\g\longrightarrow\mathbb{C}$ denote the Killing form, and set $$V^{\perp}:=\{x\in\g:\langle x,v\rangle=0\text{ for all }v\in V\}$$ for each subset $V\subseteq\g$. One has a $G$-module isomorphism defined by
\begin{equation}\label{Equation: Killing isomorphism}\g\overset{\cong}\longrightarrow\g^*,\quad x\mapsto\langle x,\cdot\rangle,\quad x\in\g.\end{equation}
At the same time, observe that the differential of $f\in\mathbb{C}[\g]$ at $x\in\g$ is an element $df_x\in\g^*$. We define the \textit{gradient} of $f$ at $x$ to be the inverse image of $df_x$ under \eqref{Equation: Killing isomorphism}, i.e. the unique element $\nabla f_x\in\g$ satisfying
$$\langle\nabla f_x,y\rangle=df_x(y)$$ for all $y\in\g$. Each subalgebra $\mathcal{A}\subseteq\mathbb{C}[\g]$ and point $x\in\g$ then determine the subspace
$$\nabla\mathcal{A}_x:=\{\nabla f_x:f\in\mathcal{A}\}\subseteq\g.$$
We also note that the Poisson structure on $\g$ amounts to the Poisson bracket
$$\{f_1,f_2\}(x):=\langle x,[(\nabla f_1)_x,(\nabla f_2)_x]\rangle,\quad f_1,f_2\in\mathbb{C}[\g],\text{ }x\in\g$$
on $\mathbb{C}[\g]$. One knows that the adjoint orbit
$$Gx:=\{\mathrm{Ad}_g(x):g\in G\}\subseteq\g$$ is the symplectic leaf through $x\in\g$, where $\mathrm{Ad}:G\longrightarrow\operatorname{GL}(\g)$ is the adjoint representation. 

Consider the subalgebra $\mathcal{I}:=\mathbb{C}[\g]^G$ of $\mathrm{Ad}$-invariant elements in $\mathbb{C}[\g]$. This subalgebra admits $\ell$ homogeneous, algebraically independent generators, and the inclusion $\mathcal{I}\subseteq\mathbb{C}[\g]$ induces the adjoint quotient map
\begin{equation}\label{Equation: Adjoint quotient}\chi:\g\longrightarrow\mathrm{Spec}(\mathcal{I})\cong\mathbb{C}^{\ell}.\end{equation} One has
$$\mathrm{ker}(d\chi_x)=(\nabla \mathcal{I}_x)^{\perp}$$ for all $x\in\g$.

It will be advantageous to fix opposite Borel subalgebras $\mathfrak{b},\mathfrak{b}^{-}\subseteq\g$, regarding the former (resp. latter) as the positive (resp. negative) Borel subalgebra. These choices yield a Cartan subalgebra $\mathfrak{h}:=\mathfrak{b}\cap\mathfrak{b}^{-}\subseteq\mathfrak{g}$, as well as sets of roots $\Delta\subseteq\h^*$, positive roots $\Delta^{+}\subseteq\Delta$, negative roots $\Delta^{-}=-\Delta^{+}\subseteq\Delta$, and simple roots $\Pi\subseteq\Delta^{+}$. One also has the nilpotent radicals $\mathfrak{u}:=[\mathfrak{b},\mathfrak{b}]$ and $\mathfrak{u}^{-}:=[\mathfrak{b}^{-},\mathfrak{b}^{-}]$ of $\mathfrak{b}$ and $\mathfrak{b}^{-}$, respectively. The decompositions
$$\g=\mathfrak{u}^{-}\oplus\mathfrak{h}\oplus\mathfrak{u}^{+},\quad\mathfrak{b}=\mathfrak{h}\oplus\mathfrak{u},\quad\mathfrak{b}^{-}=\mathfrak{h}\oplus\mathfrak{u}^{-},\quad\mathfrak{u}=\bigoplus_{\alpha\in\Delta^{+}}\g_{\alpha},\quad\text{and}\quad \mathfrak{u}^{-}=\bigoplus_{\alpha\in\Delta^{-}}\g_{\alpha}$$ then necessarily hold,
where $\g_{\alpha}\subseteq\g$ is the root space for $\alpha\in\Delta$. 
 
Note that each $\alpha\in\Delta^{+}$ determines the parabolic subalgebras
$$\mathfrak{p}_{\alpha}:=\mathfrak{b}\oplus\mathfrak{g}_{-\alpha}\quad\text{and}\quad\mathfrak{p}_{\alpha}^{-}:=\mathfrak{b}^{-}\oplus\mathfrak{g}_{\alpha}.$$ These subalgebras admit Levi decompositions of
$$\mathfrak{p}_{\alpha}=\mathfrak{l}_{\alpha}\oplus\mathfrak{u}_{\alpha}\quad\text{and}\quad \mathfrak{p}_{\alpha}^{-}=\mathfrak{l}_{\alpha}\oplus\mathfrak{u}_{\alpha}^{-},$$ where $$\mathfrak{l}_{\alpha}:=\mathfrak{g}_{-\alpha}\oplus\h\oplus\mathfrak{g}_{\alpha},\quad\mathfrak{u}_{\alpha}:=\bigoplus_{\beta\in\Delta^{+}\setminus\{\alpha\}}\mathfrak{g}_{\beta},\quad\text{and}\quad \mathfrak{u}_{\alpha}^{-}:=\bigoplus_{\beta\in\Delta^{-}\setminus\{-\alpha\}}\mathfrak{g}_{\beta}.$$  

It will be advantageous to choose root vectors $e_{\alpha}\in\g_{\alpha}\setminus\{0\}$ and $e_{-\alpha}\in\mathfrak{g}_{-\alpha}\setminus\{0\}$ for each $\alpha\in\Delta^{+}$. Let us make these choices in such a way that
$h_{\alpha}:=[e_{\alpha},e_{-\alpha}]$ satisfies $\alpha(h_{\alpha})=2$ for all $\alpha\in\Delta^+$, noting that $\{e_{\alpha},h_{\alpha},e_{-\alpha}\}$ is an $\mathfrak{sl}_2$-triple and a basis of $[\mathfrak{l}_{\alpha},\mathfrak{l}_{\alpha}]$.

\subsection{A simplification and coordinatization of the Mishchenko--Fomenko system}
Consider the regular nilpotent element $$\xi:=\sum_{\alpha\in\Pi}e_{-\alpha}\in\g$$ and the affine subspace
$$\xi+\h\subseteq\g.$$

\begin{lem}
An element of $\g$ is regular if and only if it is conjugate to an element of $\xi+\h$.
\end{lem} 

\begin{proof}
One has $\xi+\h\subseteq\greg$ \cite[Lemma 10]{KostantLie}, giving the backward implication. On the other hand, suppose that $x\in\greg$. The fibre $\chi^{-1}(\chi(x))$ contains a semisimple adjoint orbit \cite[Theorem 3]{KostantLie}. It follows that $\chi^{-1}(\chi(x))$ contains an element $y\in\h$, and we observe that $\chi(\xi+y)=\chi(y)=\chi(x)$ \cite[Lemma 11]{KostantLie}. Since $x$ and $\xi+y$ lie in $\greg$, this implies that $x$ and $\xi+y$ are $G$-conjugate \cite[Theorem 3]{KostantLie}.
\end{proof}
 
Now fix any $a\in\greg$ and choose $g\in G$ with $a':=\mathrm{Ad}_g(a)\in\xi+\h$. The algebra automorphism
$$\mathbb{C}[\g]\overset{\cong}\longrightarrow\mathbb{C}[\g],\quad f\mapsto f\circ\mathrm{Ad}_{g},\quad f\in\mathbb{C}[\g]$$
restricts to an algebra isomorphism from $\mathcal{F}_{a'}$ to $\mathcal{F}_a$. One thereby obtains an affine variety isomorphism $$\mathrm{Spec}(\mathcal{F}_a)\overset{\cong}\longrightarrow\mathrm{Spec}(\mathcal{F}_{a'})$$
and commutative diagram
\[\begin{tikzcd}[column sep=1.5em]
& \g \arrow{dl}[swap]{F_a} \arrow{dr}{F_{a'}} \\
\mathrm{Spec}(\mathcal{F}_a) \arrow{rr}{\cong} && \mathrm{Spec}(\mathcal{F}_{a'}).
\end{tikzcd}
\]
In particular, no generality is lost if one only considers Mishchenko--Fomenko systems associated to elements $a\in\xi+\h$. We shall reduce to this case on a regular basis. 

The study of Mishchenko--Fomenko systems may be simplified in the following additional way. Choose homogeneous, algebraically independent generators $f_1,\ldots,f_{\ell}$ of $\mathcal{I}$. These generators determine a variety isomorphism $\mathrm{Spec}(\mathcal{I})\cong\mathbb{C}^{\ell}$, under which the adjoint quotient becomes a map \begin{equation}\label{Equation: Coordinatized adjoint quotient}\chi=(f_1,\ldots,f_{\ell}):\g\longrightarrow\mathbb{C}^{\ell}.\end{equation} Now fix $a\in\greg$, and let $d_1,\ldots,d_{\ell}$ denote the homogeneous degrees of $f_1,\ldots,f_{\ell}$, respectively. Given any $i\in\{1,\ldots,\ell\}$, there exist unique polynomials $f_{i0},f_{i2},\ldots,f_{i(d_i-1)}\in\mathbb{C}[\g]$ such that
$$(f_i)_{\lambda,a}=f_i(a)\lambda^{d_i}+\sum_{j=0}^{d_i-1}\lambda^jf_{ij}$$
for all $\lambda\in\mathbb{C}$. One knows that $f_{i0}=f_i$ for all $i\in\{1,\ldots,\ell\}$, 
$\sum_{i=1}^{\ell}d_i=b$ \cite[Equation (1)]{Varadarajan}, and that the polynomials $\{f_{ij}:i\in\{1,\ldots,\ell\},\text{ }j\in\{0,\ldots,d_i-1\}\}$ are homogeneous, algebraically independent generators of $\mathcal{F}_a$ \cite[Section 3]{Panyushev}. In other words, the $f_{ij}$ with $j\geq 1$ extend $f_1,\ldots,f_{\ell}$ to homogeneous, algebraically independent generators $f_1,\ldots,f_{\ell},f_{\ell+1},\ldots,f_{b}$ of $\mathcal{F}_a$. These generators determine a variety isomorphism $\mathrm{Spec}(\mathcal{F}_a)\cong\mathbb{C}^b$, as a result of which the Mishchenko--Fomenko system $F_a:\g\longrightarrow\mathrm{Spec}(\mathcal{F}_a)$ becomes the map
\begin{equation}\label{Equation: Coordinatized MF} F_a=(f_1,\ldots,f_b):\g\longrightarrow\mathbb{C}^b.\end{equation} We call this the \textit{coordinatized Mishchenko--Fomenko system} resulting from the choice of $f_1,\ldots,f_{\ell}$ and chosen enumeration of $\{f_{ij}:j\geq 1\}$ as $f_{\ell+1},\ldots,f_b$. It is clear that the intergable systems-theoretic properties of a Mishchenko--Fomenko system coincide with those of any coordinatized version.

\subsection{Subregular semisimple elements}\label{Subsection: Subregular semisimple}
Recall that $x\in\g$ is called \textit{subregular} if its centralizer $\g_x\subseteq\g$ is $(\ell+2)$-dimensional. Given $\alpha\in\Delta$, the subregular elements of $\g$ lying in $\mathrm{ker}(\alpha)\subseteq\mathfrak{h}$ are given by
$$\mathrm{ker}(\alpha)^{\circ}:=\{x\in\mathrm{ker}(\alpha):\beta(x)\neq 0\text{ for all }\beta\in\Delta\setminus\{\alpha,-\alpha\}\}.$$ One readily verifies that a semisimple element of $\g$ is subregular if and only if it belongs to $$D(\alpha):=G\mathrm{ker}(\alpha)^{\circ}\subseteq\g$$ for some $\alpha\in\Delta$. The subset $D(\alpha)$ is an example of a \textit{decomposition class}, a notion introduced in \cite{BorhoCommentarii}. Such considerations imply that $D(\alpha)$ is a smooth \cite[Corollary 3.8.1 (i)]{BroerLectures}, locally closed subvariety \cite[Corollary 39.1.7 (ii)]{Tauvel} of $\g$, while it is also known that $D(\alpha)$ has codimension three in $\g$ \cite[Lemma 3.6 (iii)]{Popov}. The closures $\{\overline{D(\alpha)}\}_{\alpha\in\Delta}$ are precisely the irreducible components of $\gsing$ \cite[Lemma 3.6]{Popov}.    
\begin{lem}\label{Lemma: Component}
If $\alpha,\beta\in\Delta$ are such that $D(\alpha)\cap \overline{D(\beta)}\neq\emptyset$, then $D(\alpha)=D(\beta)$.
\end{lem}

\begin{proof}
An application of \cite[Proposition 39.2.7 (ii)]{Tauvel} tells us that $D(\beta)$ is the set of semisimple elements of $\g$ that lie in $\overline{D(\beta)}$. Since $D(\alpha)$ consists of semisimple elements, we must have $D(\alpha)\cap D(\beta)\neq\emptyset$. The conclusion $D(\alpha)=D(\beta)$ now follows from the fact that decomposition classes are equivalence classes for an equivalence relation on $\g$ \cite[Section 3.3]{BroerLectures}. 
\end{proof}

\begin{prop}\label{Proposition: Subregular semisimple tangent space}
If $x\in\g$ is subregular and semisimple, then $x$ is a smooth point of $\gsing$ and \begin{equation}\label{Equation: Tangent space equation} T_x\gsing=[\g_x,\g_x]^{\perp}.\end{equation}
\end{prop}

\begin{proof}
We first note that $x\in D(\alpha)$ for some $\alpha\in\Delta$. By Lemma \ref{Lemma: Component} and the discussion preceding it, $\overline{D(\alpha)}$ is the unique irreducible component of $\gsing$ containing $x$. This combines with the smoothness of $D(\alpha)$ to imply that $x$ is a smooth point of $\gsing$.

Since each element of $D(\alpha)$ is $G$-conjugate to an element of $\mathrm{ker}(\alpha)^{\circ}$, it suffices to prove \eqref{Equation: Tangent space equation} for $x\in\mathrm{ker}(\alpha)^{\circ}$. Note that $\g_x=\mathfrak{l}_{\alpha}$ in this case, so that $$[\g_x,\g_x]^{\perp}=\mathfrak{u}_{\alpha}^{-}\oplus\mathrm{ker}(\alpha)\oplus\mathfrak{u}_{\alpha}.$$ At the same time, the inclusions $\mathrm{ker}(\alpha)\subseteq\gsing$ and $Gx\subseteq\gsing$ force
$\mathrm{ker}(\alpha)$ and
$$T_xGx=[\g,x]=\mathfrak{u}_{\alpha}^{-}\oplus\mathfrak{u}_{\alpha}$$ to be subspaces of $T_x\gsing$. The previous two sentences imply that $$[\g_x,\g_x]^{\perp}\subseteq T_x\gsing.$$ We also know that $[\g_x,\g_x]^{\perp}=\mathfrak{u}_{\alpha}^{-}\oplus\mathrm{ker}(\alpha)\oplus\mathfrak{u}_{\alpha}$ has codimension three in $\g$. The identity \eqref{Equation: Tangent space equation} now follows from the fact that $T_x\gsing=T_x D(\alpha)$ also has codimension three in $\g$.
\end{proof}

We conclude with the following three lemmas concerning semisimple elements and the adjoint quotient \eqref{Equation: Adjoint quotient}. To this end, recall that $\g_x$ denotes the $\g$-centralizer of $x\in\g$. We write $\mathfrak{z}(\g_x)$ for the centre of this centralizer.
 
\begin{lem}\label{Lemma: nablaIxzgx}
If $x\in\g$ is semisimple, then the following statements hold:
\begin{enumerate}
\item[(i)] $\nabla \mathcal I_x = \mathfrak{z}(\g_x)$; 
\item[(ii)] $\mathrm{rank}(d\chi_x) = \dim(\mathfrak{z}(\g_x))$. 
\end{enumerate}
\end{lem}
\begin{proof}
Let $\mathcal V$ denote the $\mathcal I$-module of $G$-invariant vector fields on $\g$. By \cite[Section 1]{Richardson},  this module is generated by the gradients $\nabla f_1,\dots,\nabla f_{\ell}$. We therefore have 
$$\nabla\mathcal I_x = \{v(x): v\in \mathcal V\} =:\mathcal V_x.$$
On the other hand, \cite[Proposition 3.6.2 (ii)]{BroerLectures} shows that $\mathcal V_x = \mathfrak{z}(\g_x)$. This establishes (i), while (ii) follows from (i) and the identity
$\mathrm{ker}(d\chi_x)=(\nabla \mathcal{I}_x)^{\perp}$.
\end{proof}
 
\begin{lem}\label{Lemma: RankOfChiatx} 
If $x\in\g$ is subregular and semisimple, then the differential of $\chi$ has rank $\ell-1$ at $x$.
\end{lem}

\begin{proof}
As discussed earlier, $x\in D(\alpha)$ for some $\alpha\in\Delta$. It follows that $\g_x$ is $G$-conjugate to $\g_y$ for some $y\in\mathrm{ker}(\alpha)^{\circ}$, or equivalently that $\g_x$ is conjugate to $\mathfrak{l}_{\alpha}$. This forces $\mathfrak{z}(\g_x)$ and $\mathfrak{z}(\mathfrak{l}_{\alpha})=\mathrm{ker}(\alpha)$ to be conjugate, so that
$\dim(\mathfrak{z}(\g_x))=\ell-1$. An application of Lemma \ref{Lemma: nablaIxzgx} now completes the proof.  
\end{proof}

\begin{lem}\label{Lemma: SpanGradientsChi}
If $x\in\mathrm{ker}(\alpha)^{\circ}$ for some $\alpha\in\Delta$, then the following statements hold:
\begin{itemize}
\item[(i)] $\nabla \mathcal{I}_x=\mathrm{ker}(\alpha)$;
\item[(ii)] $\ker (d(\chi\big\vert_{\h})_x)=[\mathfrak{g}_{-\alpha},\g_{\alpha}]$.
\end{itemize} 
\end{lem}
\begin{proof}
As observed in the proof of Lemma \ref{Lemma: RankOfChiatx}, we have $\g_x =\lf_\alpha$ and $\mathfrak{z}(\lf_\alpha)=\ker(\alpha)$. 
Part (i) therefore follows from Lemma \ref{Lemma: nablaIxzgx}. 

We now verify (ii). Note that $$\ker(d\chi_x\big\vert_{\h})=(\nabla\mathcal{I}_x)^{\perp}\cap\mathfrak{h} = \ker(\alpha)^\perp\cap\h,$$ where the final instance of equality follows from (i). On the other hand, one has
$$\ker(\alpha)^\perp\cap\h=[\mathfrak{g}_{-\alpha},\g_{\alpha}].$$ These last two sentences prove (ii).   
\end{proof}

\section{The proof of Theorem \ref{Thm: Main Theorem}}\label{Section: The proof}

\subsection{A simple lemma}
Our proof of Theorem \ref{Thm: Main Theorem} is partly based on the following simple fact. 

\begin{lem}\label{Lemma: KeyLemma}
Suppose that $a\in\greg$. Assume that there exists a smooth point of $\gsing^a$ at which the differential of the restriction $F_a\big\vert_{\gsing^a}:\gsing^a\longrightarrow\mathrm{Spec}(\mathcal F_a)$ has rank $b-1$. The bifurcation diagram $\Sigma_a$ then has codimension one in $\mathrm{Spec}(\mathcal F_a)$.
\end{lem}
\begin{proof}
The assumptions imply that $\Sigma_a=F_a(\gsing^a)$ has dimension at least $b-1$. This amounts to $\Sigma_a$ having codimension at most one in $\mathrm{Spec}(\mathcal{F}_a)$, as $\mathrm{Spec}(\mathcal{F}_a)\cong\mathbb{C}^b$. On the other hand, we recall that $\Sigma_a$ is the set of critical values of $F_a$. This forces $\Sigma_a$ to have codimension at least one in $\mathrm{Spec}(\mathcal{F}_a)$. It follows that $\Sigma_a$ has codimension exactly one in $\mathrm{Spec}(\mathcal{F}_a)$.
\end{proof}

Proving that $\Sigma_a$ has codimension one thereby reduces to finding a smooth point $x\in\gsing^a$ satisfying the assumptions of Lemma \ref{Lemma: KeyLemma}. This observation is invoked at the ends of Sections \ref{Subsection: Proof i} and \ref{Subsection: Proof ii}.  
 
\subsection{The proof of Theorem \ref{Thm: Main Theorem}(i)}\label{Subsection: Proof i}
We now provide the remaining ingredients needed to prove Theorem \ref{Thm: Main Theorem}(i). To this end, suppose that $a\in\greg$ is not nilpotent. It suffices to assume that $a$ is a non-nilpotent element of $\xi+\h$, i.e. $a=\xi+y$ for some $y\in\h\setminus\{0\}$. Choose $\alpha\in\Pi$ with the property that $\alpha(y)\neq 0$. 

\begin{lem}\label{Lemma: Smooth point}
If $x\in\mathrm{ker}(\alpha)^{\circ}$, then $x$ is a smooth point of $\gsing^a$ and
$$T_x(\gsing^a)=\mathbb{C}a\oplus\mathfrak{u}_{\alpha}^{-}\oplus\mathrm{ker}(\alpha)\oplus\mathfrak{u}_{\alpha}^{+}.$$
\end{lem}

\begin{proof}
Proposition \ref{Proposition: Subregular semisimple tangent space} implies that $x$ is a smooth point of $\gsing$, and this proposition combines with the identity $\g_x=\mathfrak{l}_{\alpha}$ to give
$$T_x\gsing=[\mathfrak{l}_{\alpha},\mathfrak{l}_{\alpha}]^{\perp}=\mathfrak{u}_{\alpha}^{-}\oplus\mathrm{ker}(\alpha)\oplus\mathfrak{u}_{\alpha}^{+}.$$ It is also clear that $a\not\in\mathfrak{u}_{\alpha}^{-}\oplus\mathrm{ker}(\alpha)\oplus\mathfrak{u}_{\alpha}^{+}$, so that the tangent spaces of $\gsing$ and $x+\mathbb{C}a$ at $x$ intersect trivially. We conclude that $x$ is a smooth point of $\gsing^a$, and that
$$T_x(\gsing^a)=\mathbb{C}a\oplus\mathfrak{u}_{\alpha}^{-}\oplus\mathrm{ker}(\alpha)\oplus\mathfrak{u}_{\alpha}^{+}.$$  	
\end{proof}

\begin{lem}\label{Lemma: Rank}
If $x\in\mathrm{ker}(\alpha)^{\circ}$, then the differential of $F_a\big\vert_{Gx}:Gx\longrightarrow\mathrm{Spec}(\mathcal{F}_a)$ has rank $u-1$ at $x$. 
\end{lem}

\begin{proof}
Let $\overline{a}$ denote the projection of $a$ onto $\g_x$ with respect to the decomposition $\g=\g_x\oplus\g_x^{\perp}$. Write $(\g_x)_{\text{reg}}$ for the regular elements of $\g_x$, i.e. the elements of $\g_x$ with $\ell$-dimensional $\g_x$-centralizers. Since $\dim(Gx)=2(u-1)$, \cite[Lemma 2.1]{Panyushev2} reduces us to verifying the following: $(\C a+\C x)\cap \g_{\text{sing}} \subseteq \C x$ and $\overline{a}\in (\g_x)_{\text{reg}}$.
	
To verify the first condition, let $z,w\in\C$ be such that $za + wx\in\gsing$. Note that
$$za+wx=z\xi+(zy+wx)\in z\xi+\mathfrak{b}\subseteq\greg$$
if $z\neq 0$ \cite[Lemma 10]{KostantLie}. We conclude that $z=0$, implying that $(\C a+\C x)\cap \g_{\text{sing}} \subseteq \C x$.  
	
To establish the second condition, we observe that $\g_x=\mathfrak{l}_{\alpha}$. It follows that $\mathfrak{z}(\g_x)=\mathrm{ker}(\alpha)$ and
$$[\g_x,\g_x]=[\mathfrak{l}_{\alpha},\mathfrak{l}_{\alpha}]=\g_{-\alpha}\oplus\mathbb{C}h_{\alpha}\oplus\mathfrak{g}_{\alpha}\cong\mathfrak{sl}_2,$$ while we clearly have $\overline{a}=e_{-\alpha}+y$. The projection of $\overline{a}$ onto $[\g_x,\g_x]$ must therefore lie in $e_{-\alpha}+\mathbb{C}h_{\alpha}$, where the projection is with respect to the decomposition
$$\g_x=[\g_x,\g_x]\oplus\mathfrak{z}(\g_x).$$ Each element of $e_{-\alpha}+\mathbb{C}h_{\alpha}$ has a one-dimensional $[\g_x,\g_x]$-centralizer, implying that the $\g_x$-centralizer of $\overline{a}$ has dimension
$$1+\dim(\mathfrak{z}(\g_x))=1+(\ell-1)=\ell.$$ This is precisely the statement that $\overline{a}\in (\g_x)_{\text{reg}}$.  
\end{proof}

\begin{lem}\label{Lemma: Evaluation} 
If $x\in\mathfrak{h}$, then $(\nabla\mathcal{F}_a)_x\subseteq\mathfrak{b}^{-}$.	
\end{lem}

\begin{proof}
The main ingredient of the proof is Lemma 1.3 of \cite{Panyushev2}. This lemma states that $(\nabla\mathcal{F}_a)_x$ is the subspace of $\g$ generated by the centralizers $\g_{x+\lambda a}$, taken over all $\lambda\in\mathbb{C}$ with $x+\lambda a\in\greg$. One has $x+\lambda a\in\mathfrak{b}^{-}$ for all $\lambda\in\mathbb{C}$, so that $\g_{x+\lambda a}\subseteq\mathfrak{b}^{-}$ if $x+\lambda a\in\greg$ \cite[Corollary 2.12]{CrooksRoeserDocumenta}. These last two sentences force $(\nabla\mathcal{F}_a)_x\subseteq\mathfrak{b}^{-}$ to hold.  	
\end{proof}

One has the decomposition
\begin{equation}\label{Equation: Triple decomposition}\mathfrak{p}_{\alpha}^{-}=\mathfrak{u}_{\alpha}^{-}\oplus[\mathfrak{l}_{\alpha},\mathfrak{l}_{\alpha}]\oplus\mathrm{ker}(\alpha).\end{equation}

\begin{lem}\label{Lemma: SpanGradientsFa}
Suppose that $x\in\mathrm{ker}(\alpha)^{\circ}$. We have
$$(\nabla \mathcal F_a)_x\cap [\mathfrak{l}_\alpha,\mathfrak{l}_\alpha] =\mathbb{C}d,$$ where
$d$ is the projection of $a$ onto $[\mathfrak{l}_{\alpha},\mathfrak{l}_{\alpha}]$ with respect to \eqref{Equation: Triple decomposition}.   
\end{lem}
\begin{proof}
Lemma 1.3 of \cite{Panyushev2} tells us that $(\nabla \mathcal F_a)_x\cap\mathfrak{l}_{\alpha}$ is $\ell$-dimensional. We also know that the $(\ell-1)$-dimensional subspace $\mathrm{ker}(\alpha)$ belongs to $(\nabla \mathcal F_a)_x$, as follows from Lemma \ref{Lemma: SpanGradientsChi}(i). Since 
$$\mathfrak{l}_{\alpha} = [\mathfrak{l}_{\alpha},\mathfrak{l}_{\alpha}]\oplus\ker(\alpha),$$
these last two sentences force $(\nabla \mathcal F_a)_x\cap [\mathfrak{l}_\alpha,\mathfrak{l}_\alpha]$ to be one-dimensional. It therefore suffices to prove that $d\in(\nabla \mathcal F_a)_x$.

In light of the previous paragraph, there exists $f\in\mathcal{F}_a$ such that $\nabla f_x\in [\mathfrak{l}_{\alpha},\mathfrak{l}_{\alpha}]\setminus\{0\}$. Lemma 1.3 of \cite{Panyushev2} then implies the existence of $\lambda_1,\ldots,\lambda_k\in\mathbb{C}\setminus\{0\}$ satisfying the following conditions:
\begin{itemize}
\item $x+\lambda_i a\in\greg$ for all $i\in\{1,\ldots,k\}$;
\item $\nabla f_x=\sum_{i=1}^ky_i$ for some $y_1,\ldots,y_k\in\g$ with $y_i\in\mathfrak{g}_{x+\lambda_i a}$ for all $i\in\{1,\ldots,k\}$.
\end{itemize}  

Fix $i\in\{1,\ldots,k\}$ and note that $\g_{x+\lambda_i a}\subseteq\mathfrak{b}^{-}$ \cite[Corollary 2.12]{CrooksRoeserDocumenta}. We also note that  $$\mathfrak{b}^{-}=\mathfrak{g}_{-\alpha}\oplus[\mathfrak{g}_{-\alpha},\mathfrak{g}_{\alpha}]\oplus\mathrm{ker}(\alpha)\oplus\mathfrak{u}_{\alpha}^{-}.$$ One may therefore write 
$$y_i=\mu_i e_{-\alpha}+\nu_i h_{\alpha}+z_i$$ with $\mu_i,\nu_i\in\mathbb{C}$ and $z_i\in\mathrm{ker}(\alpha)\oplus\mathfrak{u}_{\alpha}^{-}$. Now observe that
\begin{align*}0 & = [x+\lambda_i a, y_i]\\ & =  [x+\lambda_i a, \mu_i e_{-\alpha}+\nu_i h_{\alpha}+z_i]\\ & = \underbrace{[\lambda_i d,\mu_ie_{-\alpha} + \nu_ih_\alpha]}_{\in [\mathfrak{l}_\alpha,\mathfrak{l}_\alpha]} + \underbrace{[x+\lambda_i a,z_i] + [\lambda_i(a-d), \mu_ie_{-\alpha} + \nu_ih_\alpha]}_{\in\uu_\alpha^-}.\end{align*}   
Since $\lambda_i\neq 0$, it follows that $\mu_ie_{-\alpha} + \nu_ih_\alpha$ lies in the $[\mathfrak{l}_{\alpha},\mathfrak{l}_{\alpha}]$-centralizer of $d$. This centralizer is precisely $\mathbb{C}d$, as $[\mathfrak{l}_{\alpha},\mathfrak{l}_{\alpha}]\cong\mathfrak{sl}_2$ and $d\neq 0$. On the other hand, it is clear that 
$$d = e_{-\alpha} + c h_\alpha$$
for some $c\in\mathbb{C}$. 
We conclude that
$$\mu_ie_{-\alpha} + \nu_ih_\alpha=\mu_id,$$
or equivalently 
$$y_i = \mu_i d + z_i.$$
	
The considerations discussed in the previous paragraph yield
$$\nabla f_x=\sum_{i=1}^ky_i=\left(\sum_{i=1}^k\mu_i\right)d+\sum_{i=1}^kz_i.$$ 
This combines with fact that $\nabla f_x\in [\mathfrak{l}_{\alpha},\mathfrak{l}_{\alpha}]\setminus\{0\}$ to imply that
$$\nabla f_x=\left(\sum_{i=1}^k\mu_i\right)d\quad\text{and}\quad\sum_{i=1}^k\mu_i\neq 0.$$ 
It follows that $d\in(\nabla \mathcal F_a)_x$, as desired.  
\end{proof} 

\begin{prop}\label{Proposition: Main}
If $x\in\mathrm{ker}(\alpha)^{\circ}$, then the differential of $F_a\big\vert_{\gsing^a}:\gsing^a\longrightarrow\mathrm{Spec}(\mathcal{F}_a)$ has rank $b-1$ at $x$. 
\end{prop}
\begin{proof}
We shall work with the coordinalized versions \eqref{Equation: Coordinatized adjoint quotient} and \eqref{Equation: Coordinatized MF} of $\chi$ and $F_a$, respectively. It follows that
$$F_a=(\chi,\widetilde{F_a}):\g\longrightarrow\mathbb{C}^{\ell}\oplus\mathbb{C}^{u},$$
where $\widetilde{F_a}:=(f_{\ell+1},\ldots,f_b):\g\longrightarrow\mathbb{C}^{u}$. Now write $\phi_a:\gsing^a\longrightarrow\mathbb{C}^b$ for the restriction of $F_a$ to $\gsing^a$.
Lemma \ref{Lemma: Smooth point} allows us to present $(d\phi_a)_x$ as a linear map
$$(d\phi_a)_x:\mathfrak{u}_{\alpha}^{-}\oplus\mathfrak{u}_{\alpha}^+\oplus (\C a \oplus \ker(\alpha))\longrightarrow\mathbb{C}^{\ell}\oplus\mathbb{C}^{u}.$$ Note that the direct sum decompositions of the domain and codomain give rise to the matrix representation 
\[
(d\varphi_a)_x = \begin{pmatrix}
d\chi_x\big\vert_{\uu_{\alpha}^{-}} & d\chi_x\big\vert_{\uu_{\alpha}^{+}} & d\chi_x\big\vert_{\C a \oplus \ker(\alpha)}  \\ (d\widetilde{F_a})_x\big\vert_{\uu_{\alpha}^{-}} & (d\widetilde{F_a})_x\big\vert_{\uu_{\alpha}^{+}} & (d\widetilde{F_a})_x\big\vert_{\C a \oplus \ker(\alpha)} 
\end{pmatrix}.
\]
of $(d\phi_a)_x$.

Now consider the tangent space
\begin{equation}\label{Equation: Tangent space decomposition}T_x(Gx)=[\g,x]=\mathfrak{u}_{\alpha}^{-}\oplus\oplus\mathfrak{u}_{\alpha}^+,\end{equation} and note that the adjoint quotient map is constant-valued on $Gx$. We conclude that
$$d\chi_x\big\vert_{\uu_{\alpha}^-}=0\quad\text{and}\quad d\chi_x\big\vert_{\uu_{\alpha}^+}=0.$$ On the other hand, Lemma \ref{Lemma: Evaluation} tells us that \begin{equation}\label{Equation: Useful identity}(d\widetilde{F_a})_x\big\vert_{\uu_{\alpha}^{-}}=0.\end{equation} These considerations force our matrix representation to take the form
\begin{equation}\label{Equation: Matrix representative}
(d\varphi_a)_x = \begin{pmatrix}
0 & 0 & d\chi_x\big\vert_{\C a \oplus \ker(\alpha)}  \\ 0 & (d\widetilde{F_a})_x\big\vert_{\uu_{\alpha}^{+}} & (d\widetilde{F_a})_x\big\vert_{\C a \oplus \ker(\alpha)} 
\end{pmatrix}.
\end{equation}

We again let $d$ denote the projection of $a$ onto $[\lf_\alpha,\lf_\alpha]$ with respect to \eqref{Equation: Triple decomposition}. Lemma \ref{Lemma: SpanGradientsFa} implies that $d\in (\nabla\mathcal{F}_a)_x$, and one clearly has
$$(\nabla\mathcal{F}_a)_x=\mathrm{span}\{(\nabla f_1)_x,\ldots,(\nabla f_b)_x\}.$$
We may therefore choose $c_1,\ldots,c_b\in\mathbb{C}$ so that $$f := \sum_{i=1}^b c_if_i\in\mathcal{F}_a$$ satisfies $\nabla f_x=d$. Since $d\in \g_x$, we have $\g_x^\perp \subseteq \ker(df_x).$ It follows that $$df_x(\uu_\alpha^+)=\{0\}.$$

Now recall that $$\mathrm{ker}(\alpha)=\nabla\mathcal I_x = \mathrm{span}\{(\nabla f_1)_x,\dots,(\nabla f_\ell)_x\}$$ by Lemma \ref{Lemma: SpanGradientsChi}(i), and observe that $d=\nabla f_x\not\in\mathrm{ker}(\alpha)$. These considerations imply that $c_j\neq 0$ for some $j>\ell$. By relabeling the generators $f_1,\ldots,f_b$ if necessary, we may assme that $c_{\ell+1}\neq 0$. We may therefore perform finitely many row operations on \eqref{Equation: Matrix representative} to replace the $(\ell+1)^{\text{st}}$ row with $df_x$. We also know $(d\widetilde{F_a})_x\big\vert_{\uu_{\alpha}^{+}}$ to have rank $u-1$ by \eqref{Equation: Tangent space decomposition}, \eqref{Equation: Useful identity}, and Lemma \ref{Lemma: Rank}, and have explained that $df_x(\uu_\alpha^+)=\{0\}$. These last two sentences force our new matrix representative to take the form 
\[
\begin{pmatrix}
0 & 0 & d\chi_x\big\vert_{\C a \oplus \ker(\alpha)}\\ 0 & 0 & df_x\big\vert_{\C a \oplus \ker(\alpha)}\\ 0 & A & * 
\end{pmatrix}
\]
for some invertible $(u-1)\times(u-1)$ block $A$. It will therefore suffice to prove that the block 
\[
\begin{pmatrix}
d\chi_x\big\vert_{\C a \oplus \ker(\alpha)}\\ df_x\big\vert_{\C a \oplus \ker(\alpha)}
\end{pmatrix}
\]
has rank $\ell$. Our proof will consist of verifying the following two assertions:
\begin{itemize}
\item $d\chi_x\big\vert_{\C a \oplus \ker(\alpha)}$ has rank $\ell-1$;
\item $df_x\big\vert_{\C a \oplus \ker(\alpha)}$ is not in the span of $(df_1)_x\big\vert_{\C a \oplus \ker(\alpha)},\dots,(df_\ell)_x\big\vert_{\C a \oplus \ker(\alpha)}$.
\end{itemize}   

To address the first assertion, recall that $a=\xi+y$ with $y\in\h\setminus\mathrm{ker}(\alpha)$. We may therefore define a vector space isomorphism as $$\psi:\C a\oplus\ker(\alpha)\longrightarrow \h, \quad \lambda a + w\mapsto \lambda y+w$$ for all $\lambda\in\C$ and $w\in\ker(\alpha)$. This amounts to the formula
$$\psi(\lambda a+w)=(\lambda a+w)-\lambda\xi$$ for all $\lambda\in\mathbb{C}$ and $w\in\mathrm{ker}(\alpha)$. Since $\xi\in\uu^-$, one readily deduces that
$$\chi(\psi(x))=\chi(x)\quad\text{and}\quad F_a(\psi(x))=F_a(x)$$ for all $x\in\mathbb{C}a\oplus\mathrm{ker}(\alpha)$ \cite[Lemma 11]{KostantLie}. It follows that the rank of $d\chi_x\big\vert_{\C a \oplus \ker(\alpha)}$ is equal to the rank of $d\chi_x\big\vert_{\h}$, and the latter equals $\ell-1$ by Lemma \ref{Lemma: SpanGradientsChi}(ii).

It remains only to prove that $df_x\big\vert_{\C a \oplus \ker(\alpha)}$ is not in the span of $(df_1)_x\big\vert_{\C a \oplus \ker(\alpha)},\dots,(df_\ell)_x\big\vert_{\C a \oplus \ker(\alpha)}$. To this end, write $y=ch_{\alpha}+\tilde{h}$ for some $c\in\mathbb{C}\setminus\{0\}$ and $\tilde{h}\in\mathrm{ker}(\alpha)$, so that $d = e_{-\alpha} + c h_\alpha$. It is then clear that
$$\xi-e_{-\alpha} + d\in\mathbb{C}a\oplus\mathrm{ker}(\alpha).$$
We also have $$\xi-e_{-\alpha} + d = \xi + ch_\alpha\in\ker(\alpha)^\perp,$$ while
Lemma \ref{Lemma: SpanGradientsChi}(ii) implies that
$$\ker\left((d\chi_x)\big\vert_{\C a \oplus \ker(\alpha)}\right) = \ker(\alpha)^\perp\cap (\C a\oplus\ker(\alpha)).$$
The previous two sentences imply that
$$d\chi_x\big\vert_{\C a \oplus \ker(\alpha)}(\xi-e_{-\alpha}+d) =0,$$
or equivalently that
$$(df_i)_x\big\vert_{\C a \oplus \ker(\alpha)}(\xi-e_{-\alpha}+d)=0$$ for all $i\in\{1,\ldots,\ell\}$.
On the other hand, 
$$df_x(\xi-e_{-\alpha}+d) = \langle \nabla f_x,\xi-e_{-\alpha}+d\rangle = \langle d, d + (\xi-e_{-\alpha})\rangle = \langle d,d\rangle = c^2\langle h_\alpha,h_\alpha\rangle \neq 0.$$
It follows that $df_x\big\vert_{\C a \oplus \ker(\alpha)}$ is not in the span of $(df_1)_x\big\vert_{\C a \oplus \ker(\alpha)},\dots,(df_\ell)_x\big\vert_{\C a \oplus \ker(\alpha)}$. 
\end{proof}
 
Lemma \ref{Lemma: Smooth point} and Proposition \ref{Proposition: Main} allow us to apply Lemma \ref{Lemma: KeyLemma} and finish the proof of Theorem \ref{Thm: Main Theorem}(i). 

\subsection{The proof of Theorem \ref{Thm: Main Theorem}(ii)}\label{Subsection: Proof ii}
Section 6.1 of \cite{CrooksRoeserDocumenta} explains that $\Sigma_a$ is a singleton if $\g=\mathfrak{sl}_2$ and $a\in\greg$ is nilpotent. One immediate consequence is that $\Sigma_a$ has codimension two in $\mathrm{Spec}(\mathcal{F}_a)\cong\mathbb{C}^2$. In what follows, we establish that $\Sigma_a\subseteq\mathrm{Spec}(\mathcal{F}_a)$ has codimension one if $\g=\mathfrak{sl}_3$ and $a\in\greg$ is nilpotent. The assertions in the previous two sentences then combine with \cite[Proposition 5.2]{CrooksRoeserDocumenta} to imply Theorem \ref{Thm: Main Theorem}(ii).

In what follows, we replace the Killing form on $\mathfrak{sl}_3$ with the trace form, i.e. $$\langle x,y\rangle:=\mathrm{tr}(xy),\quad x,y\in\mathfrak{sl}_3.$$
Now consider the standard Cartan subalgebra $\h\subseteq\mathfrak{sl}_3$ of diagonal matrices, as well as the standard Borel subalgebra $\mathfrak{b}\subseteq\mathfrak{sl}_3$ of upper-triangular matrices. The resulting simple roots $\alpha,\beta\in\mathfrak{h}^*$ are then specified by their corresponding root spaces
$$(\mathfrak{sl}_3)_{\alpha}=\mathrm{span}\left\{e_{\alpha}:=\begin{pmatrix}
0 & 1 & 0 \\ 0 & 0 & 0 \\ 0 & 0 & 0
\end{pmatrix} \right\}\quad\text{and}\quad(\mathfrak{sl}_3)_{\beta}=\mathrm{span}\left\{e_{\beta}:=\begin{pmatrix}
0 & 0 & 0 \\ 0 & 0 & 1 \\ 0 & 0 & 0
\end{pmatrix} \right\}.$$
It follows that
$$(\mathfrak{sl}_3)_{-\alpha}=\mathrm{span}\left\{e_{-\alpha}:=\begin{pmatrix}
0 & 0 & 0 \\ 1 & 0 & 0 \\ 0 & 0 & 0
\end{pmatrix} \right\}\quad\text{and}\quad(\mathfrak{sl}_3)_{-\beta}=\mathrm{span}\left\{e_{-\beta}:=\begin{pmatrix}
0 & 0 & 0 \\ 0 & 0 & 0 \\ 0 & 1 & 0
\end{pmatrix} \right\}.$$
Extensive use will be made of the Levi subalgebra $$\lf_\alpha=(\mathfrak{sl}_3)_{-\alpha}\oplus\h\oplus(\mathfrak{sl}_3)_{\alpha}$$ and its centre \begin{equation}\label{Equation: Centre}\mathfrak{z}(\lf_\alpha)=\mathrm{ker}(\alpha)=\mathrm{span}\left\{\begin{pmatrix}
1 & 0 & 0 \\ 0 & 1 & 0 \\ 0 & 0 & -2
\end{pmatrix} \right\}.\end{equation}

\begin{lem}\label{Lemma: xsmoothsing}    
The element $x= e_\alpha$ belongs to the smooth locus of $(\sln_3)_{\text{sing}}$ and satisfies 
\[
T_x((\mathfrak{sl}_3)_{\text{sing}})  = [\mathfrak{sl}_3,x] \oplus \mathfrak{z}(\lf_\alpha). 
\]

\end{lem}
\begin{proof}
It follows from Proposition 3.6 and Remark 3.7 in \cite{Popov} that $(\sln_3)_{\text{sing}}$ is irreducible and has codimension three in $\mathfrak{sl}_3$. We are thereby reduced to proving that $T_x((\mathfrak{sl}_3)_{\text{sing}})  = [\mathfrak{sl}_3,x] \oplus \mathfrak{z}(\lf_\alpha)$, and that the right-hand side has codimension three in $\mathfrak{sl}_3$.

The polynomials $f_1,f_2:\mathfrak{sl}_3\longrightarrow\mathbb{C}$ defined by
$$f_1(y)=\tr(y^2)\quad\text{and}\quad f_2(y)=\tr(y^3)$$ are algebraically independent generators of the subalgebra $\mathcal{I}\subseteq\mathbb{C}[\mathfrak{sl}_3]$. It follows that
\[
(\mathfrak{sl}_3)_{\text{sing}} = \{y\in \sln_3: (df_1)_y\wedge(df_2)_y = 0\},
\]
e.g. by \cite[Theorem 9]{KostantLie}.
On the other hand, 
\[
(df_1)_y(z) = 2\tr(yz) \quad\text{and}\quad (df_2)_y(z) = 3\tr(y^2z)
\] 
for all $y,z\in\mathfrak{sl}_3$.
These amount to the statements 
\[
(\nabla f_1)_y = 2y \quad\text{and}\quad (\nabla f_2)_y = 3y^2 - \tr(y^2)\mathrm{I}_3
\]
for all $y\in\mathfrak{sl}_3$, where $\mathrm{I}_3$ is the $3\times 3$ identity matrix.
We conclude that
\[
(\mathfrak{sl}_3)_{\text{sing}} = \{y\in\sln_3: y\wedge(y^2 - \frac{1}{3}\tr(y^2)\mathrm{I}_3) = 0\},
\] which combines with the identity $x^2=0$ to give
\begin{equation}\label{Equation: Defining}
T_x((\mathfrak{sl}_3)_{\text{sing}}) = \{z\in\sln_3\colon x\wedge(xz+zx-\frac{2}{3}\tr(zx)\mathrm{I}_3) = 0\}.
\end{equation}
Now suppose that $z\in\mathfrak{sl}_3$ and write 
\[
z = \begin{pmatrix}
z_1 & z_2 & z_3 \\ z_4 & z_5 & z_6 \\ z_7 & z_8 & -(z_1+z_5)
\end{pmatrix}, \qquad z_1,\dots,z_8\in\C. 
\]
Note that the trace-free part of $xz+zx$ is
\[
xz + zx - \frac{2}{3}\tr(xz)\mathrm{I}_3 = \begin{pmatrix}
\tfrac{1}{3}z_4 & z_5+z_1 & z_6 \\ 0 & \tfrac{1}{3}z_4 & 0 \\ 0 & z_7 & -\frac{2}{3}z_4
\end{pmatrix}.
\]
By \eqref{Equation: Defining}, $z\in T_x((\mathfrak{sl}_3)_{\text{sing}})$ if and only if $x$ and this trace-free part are linearly dependent. It is now straightforward to verify that
\begin{equation}\label{Equation: Tangent presentation}
T_x((\mathfrak{sl}_3)_{\text{sing}}) = \left\{ \begin{pmatrix}
z_1 & z_2 & z_3 \\ 0 & z_5 & 0 \\ 0 & z_8 & -(z_1+z_5)
\end{pmatrix}: z_1,z_2,z_3,z_5,z_8\in\C  \right\}.
\end{equation}
At the same time, one readily checks that
\begin{equation}\label{Eq: xsl3}
[\sln_3,x] = \mathrm{span}\left\{\begin{pmatrix}
1 & 0 & 0 \\ 0 & -1 & 0 \\ 0 & 0 & 0
\end{pmatrix}, \begin{pmatrix}
0& 1 & 0 \\ 0 & 0 & 0 \\ 0 & 0 & 0
\end{pmatrix}, \begin{pmatrix}
0& 0& 1 \\ 0 & 0 & 0 \\ 0 & 0 & 0
\end{pmatrix} , \begin{pmatrix}
0& 0 & 0 \\ 0 & 0 & 0 \\ 0 & 1 & 0
\end{pmatrix} \right\}.
\end{equation}
The desired result now follows easily from \eqref{Equation: Centre}, \eqref{Equation: Tangent presentation}, and \eqref{Eq: xsl3}.  
\end{proof}

Now consider the nilpotent element  
\[
\xi := e_{-\alpha} + e_{-\beta} =  \begin{pmatrix}
0 & 0 & 0 \\ 1 & 0 & 0 \\ 0 & 1 & 0
\end{pmatrix}\in(\mathfrak{sl}_3)_{\text{reg}},
\]
and recall that the critical points of $F_{\xi}:\mathfrak{sl}_3\longrightarrow\mathrm{Spec}(\mathcal{F}_{\xi})$ are given by
$$(\mathfrak{sl}_3)_{\text{sing}}^{\xi}=(\mathfrak{sl}_3)_{\text{sing}}+\mathbb{C}\xi\subseteq\mathfrak{sl}_3.$$ 
Let us also note that $\xi\not\in T_x((\mathfrak{sl}_3)_{\text{sing}})$ by \eqref{Equation: Tangent presentation}. We thus have the following immediate consequence of Lemma \ref{Lemma: xsmoothsing}.

\begin{cor}\label{Cor: xsmoothsingxi}
The element $x=e_\alpha$ belongs to the smooth locus of $(\sln_3)_{\text{sing}}^\xi$ and  satisfies 
\[
T_x((\mathfrak{sl}_3)_{\text{sing}}^{\xi}) = [\mathfrak{sl}_3,x] \oplus \mathfrak{z}(\lf_\alpha) \oplus \C\xi. 
\]
\end{cor}

In keeping with \eqref{Equation: Coordinatized MF} and the discussion preceding it, $F_{\xi}:\mathfrak{sl}_3\longrightarrow\mathrm{Spec}(\mathcal{F}_{\xi})$ may be coordinatized as 
\[ F_{\xi}:\mathfrak{sl}_3\longrightarrow\mathbb{C}^5,\quad
y\mapsto (\tr(y^2), \tr(y^3), 2\tr(\xi y), 3\tr(\xi^2y), 3\tr(\xi y^2)),\quad y\in\mathfrak{sl}_3.
\]
It follows that
\begin{equation}\label{Eq: DFxi}
(dF_\xi)_y(z) = (2\tr(yz), 3\tr(zy^2), 2\tr(\xi z), 3\tr(\xi^2z), 3\tr((\xi y + y\xi)z)).
\end{equation}
for all $y,z\in\mathfrak{sl}_3$.

\begin{prop}\label{Proposition: Qmain}
If $x=e_{\alpha}$, then the differential of $F_\xi\big\vert_{(\mathfrak{sl}_3)_{\mathrm{sing}}^{\xi}}:(\mathfrak{sl}_3)_{\mathrm{sing}}^{\xi}\longrightarrow\mathrm{Spec}(\mathcal{F}_\xi)$ has rank four at $x$. 
\end{prop}
\begin{proof}
Corollary \ref{Cor: xsmoothsingxi} combines with \eqref{Equation: Centre} and \eqref{Eq: xsl3} to imply that 
\[T_x((\mathfrak{sl}_3)_{\text{sing}}^{\xi}) = \mathrm{span}\left\{\begin{pmatrix}
1 & 0 & 0 \\ 0 & -1 & 0 \\ 0 & 0 & 0
\end{pmatrix}, \begin{pmatrix}
0& 1 & 0 \\ 0 & 0 & 0 \\ 0 & 0 & 0
\end{pmatrix}, \begin{pmatrix}
0& 0& 1 \\ 0 & 0 & 0 \\ 0 & 0 & 0
\end{pmatrix} , \begin{pmatrix}
0& 0 & 0 \\ 0 & 0 & 0 \\ 0 & 1 & 0
\end{pmatrix},\begin{pmatrix}
1& 0 & 0 \\ 0 & 1 & 0 \\ 0 & 0 & -2
\end{pmatrix}, \begin{pmatrix}
0& 0 & 0 \\ 1 & 0 & 0 \\ 0 & 1 & 0
\end{pmatrix} \right\}\]
It follows that any $z\in T_x((\mathfrak{sl}_3)_{\text{sing}}^{\xi})$ has the form
\[
z= \begin{pmatrix}
z_1+z_2& z_3 & z_4 \\ z_5 & -z_1+z_2 & 0 \\ 0 & z_5+z_6 & -2z_2
\end{pmatrix}.
\]
for $z_1,\ldots,z_6\in\mathbb{C}$.
Taken together with \eqref{Eq: DFxi} and the identity $x^2 =0$, this yields
\[ 
(dF_\xi)_x(z) = (2z_5, 0, 2z_3, 3z_4, 6z_2).
\]
We conclude that 
\[
\ker ((dF_\xi)_x)\cap T_x((\mathfrak{sl}_3)_{\text{sing}}^{\xi}) = \left\{\begin{pmatrix}
z_1& 0 & 0 \\ 0 & -z_1 & 0 \\ 0 & z_6 & 0
\end{pmatrix}: z_1,z_6\in\C\right\}
\] is a two-dimensional subspace of the six-dimensional vector space $T_x((\mathfrak{sl}_3)_{\text{sing}}^{\xi})$. The restriction of $(dF_\xi)_x$ to $T_x((\mathfrak{sl}_3)_{\text{sing}}^{\xi})$ must therefore have rank four.
\end{proof} 

Now note that $b-1=4$ in our case of $\mathfrak{g}=\mathfrak{sl}_3$. Lemma \ref{Lemma: KeyLemma}, Corollary \ref{Cor: xsmoothsingxi}, and  Proposition \ref{Proposition: Qmain} therefore allow us to conclude that $\Sigma_\xi$ has codimension one in $\mathrm{Spec}(\mathcal{F}_{\xi})$. One also knows that any nilpotent element $a\in(\mathfrak{sl}_3)_{\text{reg}}$ is conjugate to $\xi$. It follows that $\Sigma_a$ has codimension one in $\mathrm{Spec}(\mathcal{F}_{a})$ if $a\in(\mathfrak{sl}_3)_{\text{reg}}$ is nilpotent. As discussed at the beginning of Section \ref{Subsection: Proof ii}, this completes the proof of Theorem \ref{Thm: Main Theorem}(ii).
 
\bibliographystyle{acm} 
\bibliography{Bif}
\end{document}